\documentclass{amsart}
\usepackage{amsfonts,amssymb,amscd,amsmath,enumerate,verbatim,calc}
\everymath{\displaystyle}
\usepackage{graphicx}
\usepackage{epstopdf}
\usepackage{epsfig}
\usepackage{graphicx}
\usepackage{float}
\usepackage{subcaption}
\usepackage{caption}
\newcommand{\R}{\mathbb{R}}

\theoremstyle{plain}

\newtheorem{theorem}{Theorem}[section]
\newtheorem{corollary}[theorem]{Corollary}
\newtheorem{lemma}[theorem]{Lemma}
\newtheorem{proposition}[theorem]{Proposition}

\theoremstyle{definition}
\newtheorem{definition}[theorem]{Definition}

\newtheorem{remark}[theorem]{Remark}
\newtheorem{example}[theorem]{Example}

\begin{document}
\title[Fractional Differential and Integral]{A New Fractional Derivative and its Fractional Integral with some applications}
\author{Fahed~Zulfeqarr}
\address{Department of Applied Mathematics, Gautam Buddha University, Greater Noida, Uttar Pradesh 201312, India}
\email{fahed@gbu.ac.in}
\author{Amit~Ujlayan}
\address{Department of Applied Mathematics, Gautam Buddha University, Greater Noida, Uttar Pradesh 201312, India}
\email{amitujlayan@gbu.ac.in}
\author{Priyanka Ahuja}
\address{Department of Applied Mathematics, Gautam Buddha University, Greater Noida, Uttar Pradesh 201312, India}
\email{pahuja836@gmail.com}
%\thanks{.....}

\begin{abstract}
A new derivative, called deformable derivative, is introduced here which is equivalent to ordinary derivative in the sense that one implies other. The deformable derivative is defined using limit approach like that of ordinary one but with respect to a parameter varying over unit interval. Thus it could also be regarded as a fractional derivative. Reason of calling it as deformable derivative is because of its intrinsic property of continuously deforming function to derivative. This is substantiated by its linear connection to function and its derivative. Besides discussing some of its basic properties, we discover the forms of Rolle's, Mean Value and Taylor's theorems. The fundamental theorem of calculus for this fractional derivative could be taken as definition of its fractional integral. As a theoretical application some fractional differential equations are solved.  
\end{abstract}
\date{\today}
%\footnote{file V2.3}
\maketitle
\section*{Introduction}
\noindent A little variation to looking some mathematical concept may sometimes sheds light on some hidden facts. For instance, continuity and differentiability are both limits concepts but defined differently. However latter one tells more geometric facts about function than that of former one. Similarly generalization of any mathematical concept besides being a great source of motivation in its own not just simplifies various intricate facts pertaining to it but extend its applicability to a wider class of problems. For instance, the first proof of prime number theorem, a very popular theorem in real numbers, goes through the techniques of complex analysis. Likewise there have been several generalizations of the notion of derivative to fractional derivative since the time L'Hospital first asked this question to Leibniz  in his letter \cite{Leibniz} in $1695$ about a meaningful interpretation of $\;\frac{d^{1/2}y}{d x^{1/2}}$. Various types of fractional derivative have been introduced so far. Ironically most of the definitions of fractional derivative carry integral form. However only few grab attention of mathematicians and become popular in world of fractional calculus namely Riemann-Liouville, Caputo, Hadamard,  Grunwald-Letnikov and Riesz fractional derivatives. To gain a good insight into fractional calculus, reader is advised to go through \cite{Miller,Podlubny,Oldham}.\\

As pointed most of definitions of  fractional derivative use the integral form. In contrast, R. Khalil  introduced a limit based definition of fraction derivative \cite{Khalil70} in $2014$, calling it as conformable fraction derivative in analogy to that of standard one. However his definition lacks to include zero and negative numbers. We hereby define a new derivative, referred to as deformable derivative, which is much simpler than that of Khalil's one and overcomes not only this shortcoming but ranges over a wider class of functions.\\

\begin{definition} Let $f(t)$ be real valued function defined on interval $(a,b)$. For a given number $\alpha$, 
$0\leqslant\alpha\leqslant 1$, we define {\it deformable derivative} by the following limit: 
\begin{equation*}
\lim_{\epsilon\rightarrow0}\frac{(1+\epsilon\beta)f(t+\epsilon\alpha)-f(t)}{\epsilon},\quad\mbox{where}\;\;\alpha+\beta=1.\tag{*}
\end{equation*}
 If this limit exists, we denote it by $D^\alpha f(t)$. 
\end{definition}
\begin{remark}
One can note that the definition (*) is compatible with $\alpha= 0,1$. For, if $\alpha=0$, $D^0 f(t)=f(t)$ which is the usual convention and if $\alpha=1$,
$Df(t)=f'(t)$. Therefore it can be deemed as a new fractional derivative with respect to parameter $\alpha$. Throughout the article until unless specified we assume that $0<\alpha\leqslant 1.$ The deformable derivative with a given $\alpha$ is sometimes referred to as $\alpha$-derivative as well. 
\end{remark} 

\indent In first section we derive a formula connecting both $\alpha$-derivative and ordinary derivative, viz. $D^\alpha f(t)=\beta f(t)+\alpha D f(t)$ and end it up with a conclusion that for a function, $\alpha$-differentiability is same as differentiability in the sense that existence of one implies that of other. Section two focuses on some of basic properties of deformable  derivative. We illustrate geometrically the behavior of operator $D^\alpha$ on some elementary functions. These examples exhibit how deformable  derivative sits between function and its derivative. Section three discusses the form of extension of Rolle's, Mean-Value and Taylor's theorems in the context of deformable derivative. Next section introduces the integral fractional operator defined as follows: 
\[
I^{\alpha}_{a}f(t)=\frac{1}{\alpha}e^{\frac{-\beta }{\alpha}t}\int_{a}^{t}e^{\frac{\beta }{\alpha}x}f(x)dx
\]
for continuous function $f$ over the interval $(a,b)$. Some basic properties of this fractional integral operator $I^{\alpha}_a$ are also discussed. In the last section we solve some fractional differential equations using these operators.\\

\section{\textbf{Preliminary Results}}
This section exhibits a relation of deformable derivative with function and its derivative. This is where the name deformable derivative is taken from. The relation further exposes the interesting fact about the graph of deformable derivative lying linearly between that of function and derivative. 
\bigskip

The first result is quite natural and asserts that {\it differentiability implies $\alpha$-differentiability}. The proof connects both operators.

\begin{theorem}\label{diff.1}
A differentiable $f$ at a point $t\in (a,b)$ is always $\alpha$- differentiable at that point for any $\alpha$. Moreover in this case we have 
\begin{equation}
D^\alpha f(t)=\beta f(t)+\alpha D f(t),\quad\mbox{where\;\;} \alpha+\beta=1.\tag{**}
\end{equation}
\end{theorem}	
\begin{proof}
By definition we have
\begin{eqnarray*}
D^{\alpha}f(t)&=&\lim\limits_{\epsilon\rightarrow0}\frac{\left(1+\epsilon\beta\right)f(t+\alpha \epsilon)-f(t)}{\epsilon}\\
&=& \lim\limits_{\epsilon\rightarrow 0}\left(\frac{f(t+\alpha \epsilon)-f(t)}{\epsilon}-\beta f(t+\alpha \epsilon)\right)\\
&=& \alpha\cdot D f(t)-\beta\cdot \lim\limits_{\epsilon\rightarrow0}f(t+\alpha \epsilon).
%\Rightarrow\quad f'(t)&=&\frac{1}{\alpha} \left( D^\alpha f(t)-\beta\cdot f(t)\right)
\end{eqnarray*}
Both the terms exist as $f$, being differentiable at $t$ is continuous as well. Hence theorem follows.
\end{proof}
\bigskip

The second result of the section that is also natural is about {\it Does $\alpha$-differentiability imply continuity} ? Answer is affirmative. However to prove the claim we need an auxiliary result concerning to locally bounded function. A function is said to be {\it locally bounded} at a point if it is bounded in some neighbourhood of that point. Precisely a function $f$ defined on $(a,b)$ is locally bounded at $t$ if there are positive numbers $M\;\&\;\delta$ such that
\[
\bigl\lvert f(t+\epsilon)\bigr\rvert\leqslant M\quad \mbox{whenever}\;\;\bigl\lvert\epsilon\bigr\rvert<\delta.
\]
Here $\delta$ is chosen sufficiently small so that $~t+\epsilon \in(a,b)$.\\

\begin{lemma}\label{locally.bounded}
Suppose $f$ is $\alpha$-differentiable in $(a,b)$ with respect to some $\alpha$. Then $f$ is locally bounded there.
\end{lemma}
\begin{proof}
Suppose $f$ is $\alpha$-differentiable at $t$, there exist a number $\delta>0$ such that
\begin{eqnarray*}
\bigl\lvert (1+\epsilon\beta)f(t+\epsilon\alpha)-f(t)-\epsilon\cdot D^\alpha f(t)\bigr\rvert \leqslant \bigl\lvert\epsilon\bigr\rvert,\quad &&\mbox{whenever}\;\;\bigl\lvert\epsilon\bigr\rvert<\delta\\
\Rightarrow\quad \bigl\lvert (1+\epsilon\beta)f(t+\epsilon\alpha)\bigr\rvert \leqslant |\epsilon|+\bigl\lvert f(t)+\epsilon\cdot D^\alpha f(t)\bigr\rvert,\quad &&\mbox{whenever}\;\;\bigl\lvert\epsilon\bigr\rvert<\delta\\
\leqslant |\epsilon|\left(1+\bigl\lvert D^\alpha f(t)\bigr\rvert \right)+\bigl\lvert f(t)\bigr\rvert,\quad &&\mbox{whenever}\;\;\bigl\lvert\epsilon\bigr\rvert<\delta\\
\Rightarrow\quad \bigl\lvert f(t+\epsilon\alpha)\bigr\rvert \leqslant \frac{|\epsilon|\left(1+\bigl\lvert D^\alpha f(t)\bigr\rvert \right)+\bigl\lvert f(t)\bigr\rvert}{\bigl\lvert 1+\beta \epsilon\bigr\rvert} \quad &&\mbox{whenever}\;\;\bigl\lvert\epsilon\bigr\rvert<\delta
\end{eqnarray*}
This yields that $f$ is locally bounded at $t$.
 \end{proof}
 
 The next theorem establish the claim that $\alpha$-differentiability implies continuity.\\
 
\begin{theorem}\label{diff.2}
 Let $f$ be $\alpha$-differentiable at a point $t$ for some $\alpha$. Then it is continuous there.
\end{theorem}	
	
\begin{proof}
For continuity, it suffices to prove the following: 
	
\[
\lim\limits_{\epsilon\rightarrow0}\left(f(t+\epsilon\alpha)-f(t)\right)=0.
\]
	
The left hand side can also be written as:
	
\begin{eqnarray*}
&&\lim\limits_{\epsilon\rightarrow0}\frac{\left(1+\epsilon\beta\right)f(t+\epsilon\alpha)-f(t)-\epsilon\beta f(t+\epsilon\alpha)}{\epsilon}\epsilon\\
&&=\lim\limits_{\epsilon\rightarrow0}\left(\frac{\left(1+\epsilon\beta\right)f(t+\epsilon\alpha)-f(t)}{\epsilon}\cdot\epsilon-\beta\epsilon\cdot f(t+\epsilon\alpha)\right)\\
&&= D^\alpha f(t)\cdot 0-\beta\lim\limits_{\epsilon\rightarrow0}\epsilon f(t+\epsilon\alpha)=0,\quad(\mbox{by hypothesis})\\
&&=-\beta\lim\limits_{\epsilon\rightarrow0}\epsilon f(t+\epsilon\alpha)=0.\quad(\mbox{by using lemma \ref{locally.bounded}})
\end{eqnarray*}
This completes theorem.
\end{proof} 

A strong version of the theorem \ref{diff.2} as an easy consequence is given in the following corollary: 

\begin{corollary}\label{diff.3}
 An $\alpha$-differentiable function $f$ defined in $(a,b)$ is differentiable as well.
\end{corollary} 
\begin{proof}
For the existence of derivative we use its definition
\begin{eqnarray*}
Df(t)&=& \frac{1}{\alpha}\cdot\lim\limits_{\epsilon\rightarrow 0}\frac{f(t+\alpha \epsilon)-f(t)}{\epsilon}\\
&=&\frac{1}{\alpha}\cdot\lim\limits_{\epsilon\rightarrow 0}\frac{\left(1+\epsilon\beta\right)f(t+\alpha \epsilon)-f(t)-\epsilon\beta f(t+\alpha \epsilon)}{\epsilon}\\
\Rightarrow \quad D f(t) &=&  \frac{1}{\alpha}\left( \lim\limits_{\epsilon\rightarrow 0}\frac{\left(1+\epsilon\beta\right)f(t+\alpha \epsilon)-f(t)}{\epsilon}-\beta\cdot \lim\limits_{\epsilon\rightarrow0}f(t+\alpha \epsilon)\right)
\end{eqnarray*}
By using hypothesis and theorem \ref{diff.2}, we get the result done. 
\end{proof}
 
We summarise all these by saying that two concepts $\alpha$-differentiability and differentiability of a function defined in $(a,b)$ are equivalent in the sense that one implies other. We write it as separate theorem.\\
 
\begin{theorem}\label{diff.4}
Let $f$ be defined in $(a,b)$. For any $\alpha$, $f$ is $\alpha$-differentiable if and only if it is differentiable.
\end{theorem}
 
\begin{remark}\label{rem.1.1}
We make a remark here that over the interval $(a,b)$, the existence of $\alpha$-derivative with respect to one particular value of $\alpha>0$ is sufficient for the existence of $\alpha$-derivative with respect to all other values of $\alpha$. 
\end{remark}

Though the most important case for $\alpha$ is when $\alpha\in [0,1]$  but what happens if $\alpha \in (n,n+1]$ for any natural number $n$ and what will be definition?
\\

\begin{definition}
Suppose $f$ is $n$-times differentiable at $t\in (a,b).$  For given $\alpha\in(n,n+1]$, we extend deformable derivative in a very natural way and define it by the following limit:
\[
D^\alpha f(t)=\lim\limits_{\epsilon\rightarrow 0}\frac{\left(1+\epsilon\{\beta\}\right))D^{n}f\left(t+\epsilon\{\alpha\}\right)-D^{n}f(t)}{\epsilon}
\]
where $\{\alpha\}$ is the fractional part of $\alpha$ and $\{\alpha\}+\{\beta\}=1$.
\end{definition}

As the consequence of the above definition, if $f^{(n+1)}$ exists, we have
\[
D^\alpha f(t)=\{\beta\}D^{n}f(t)+\{\alpha\}D^{n+1}f(t).
\]

\section{\textbf{Basic Properties of Deformable Derivative}}
\noindent Apart from discussing some fundamental properties of deformable derivative like linearity and commutativity the section deals with fundamental theorems: Rolle's, Mean-Value and Taylor's theorems. The geometric illustration of $D^{\alpha}f$ for some elementary functions $f$ is also given.

\begin{theorem}\label{properties.1}
The operator $D^{\alpha}$ possesses the following properties:
\begin{itemize}
\item [$($\rm a$)$] Linearity: $D^{\alpha} \left(a f+b g\right)=a D^{\alpha} f+b D^{\alpha} g$
\item [$($\rm b$)$]  Commutativity: $D^{\alpha_1}\cdot D^{\alpha_2}=D^{\alpha_2}\cdot D^{\alpha_1}$. Hence in general, for any positive integer $n\in \mathbb N$ we have $D^\alpha \cdot D^n=D^n\cdot D^\alpha$
\item [$($\rm c$)$] For a constant function $k$, \;$D^\alpha(k)=\beta k$.
\item [$($\rm d$)$] $D^{\alpha} (f\cdot g)= \left(D^{\alpha}f\right)\cdot g+\alpha f\cdot D g.$ Hence $D^\alpha$ does not obey the Leibniz rule. 
\end{itemize}
\end{theorem}
\begin{proof}
Linearity is evident from definition. Commutativity follows readily by noticing the symmetries in the expression below:
\[
D^{\alpha_{1}}\left(D^{\alpha_{2}}f\right)
=\beta_{1}\beta_{2}f+(\alpha_{1}\beta_{2}+\alpha_{2}\beta _{1})D f+\alpha_{1}\alpha_{2}D^2 f, 
\]
where $\alpha_i+\beta_i=1$ for $i=1,2$. Using the relation: $D^\alpha =\beta I+\alpha D$, the third and fourth can be easily established. Violation of Leibniz rule in part (d) motivates to regard $D^\alpha$ as fractional derivative. Readers are advised to see \cite{Tarasov}.
\end{proof}

Most of the familiar functions behave well with respect to differentiation so their deformable derivatives can be obtained from expression (**) in theorem \ref{diff.1}. We list out the deformable derivatives of some elementary functions in the following proposition. 

\begin{proposition}
$~$
\begin{enumerate}
\item $ D^{\alpha} (t^r )=\beta t^r+r\alpha t^{r-1}$, \quad $r\in \R$.
\item $D^{\alpha}(e^t)=e^t.$
\item $D^{\alpha}(\sin t)=\beta\sin t+\alpha\cos t$. 
\item $D^{\alpha}(\log t)=\beta\log t+\frac{\alpha}{t}$, $t>0$.
\end{enumerate}
\end{proposition}

It is intuitively clear that the operator $D^\alpha$ is continuous with respect to parameter $\alpha$. However we leave it for reader to prove. Instead we focus on the  geometric realization of the ideal with some examples.  The following figures  depict not only continuity phenomenon but also explain its nature of deforming function to its derivative.  

\begin{figure}[H]
\centering
\begin{subfigure}[b]{0.45\textwidth}
\includegraphics[width=\textwidth]{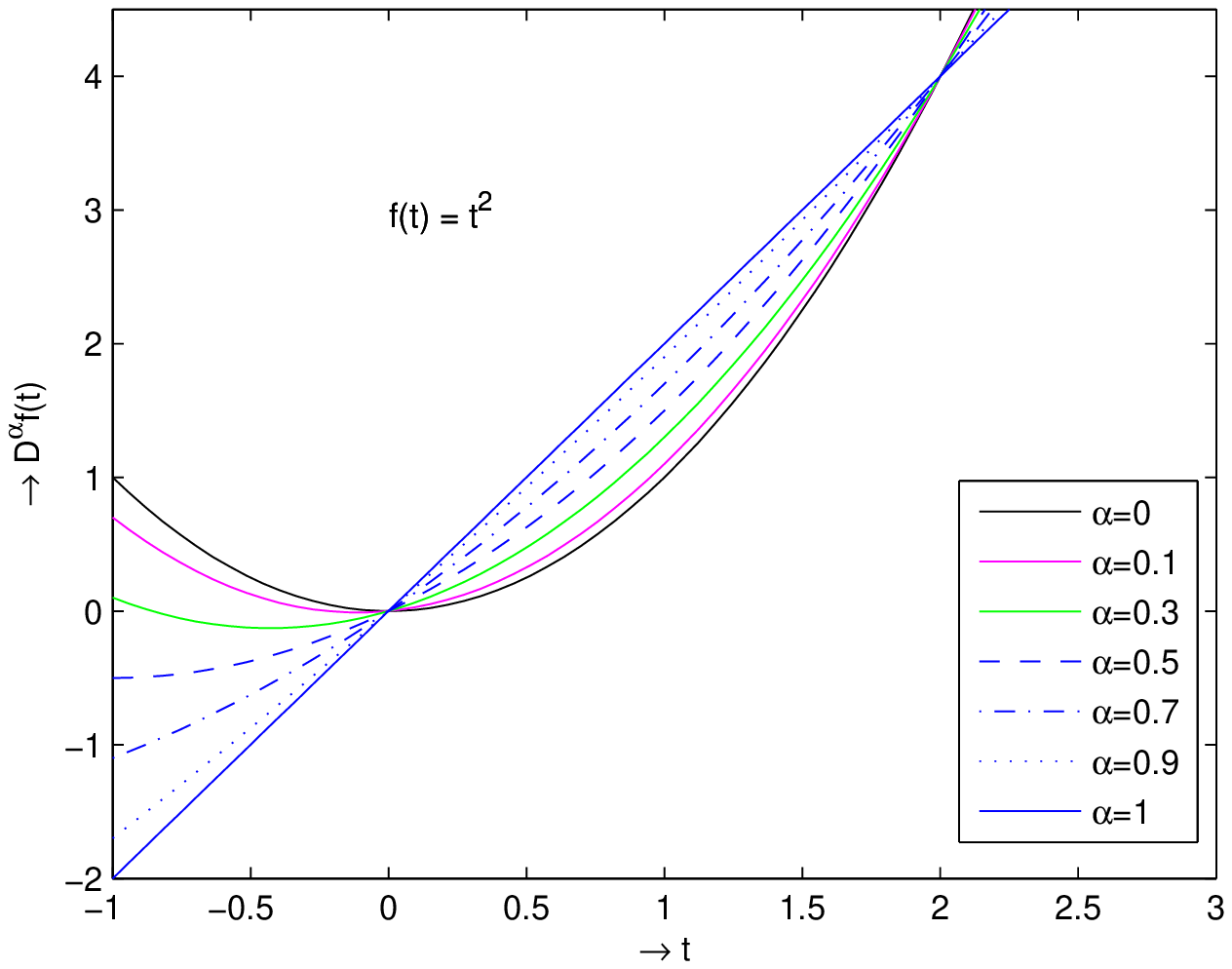}
\caption{$D^\alpha$ operating on  $t^2$}
\label{fig:1}
\end{subfigure}
~ 
\begin{subfigure}[b]{0.45\textwidth}
\includegraphics[width=\textwidth]{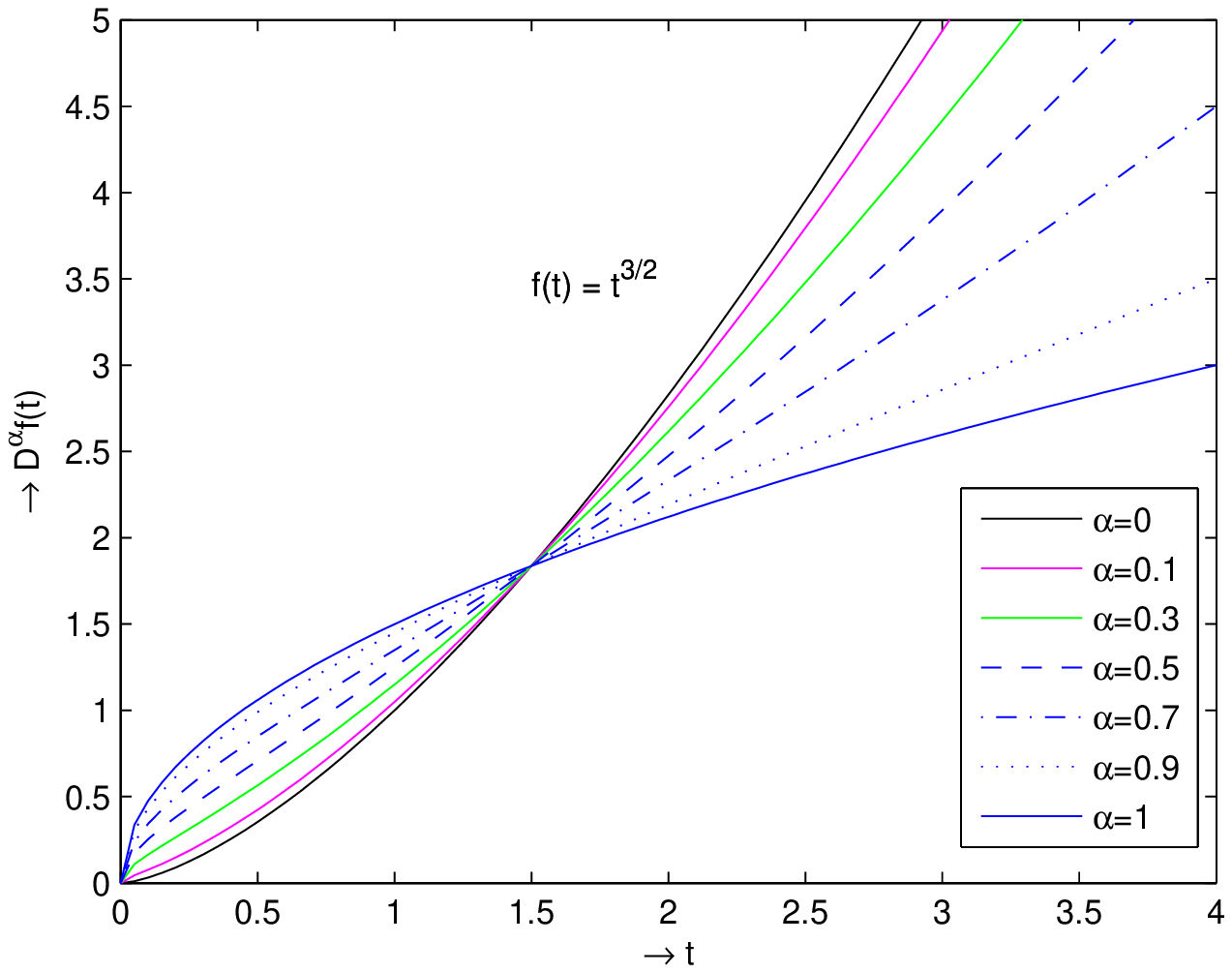}
\caption{$D^\alpha$ operating on  $t^{3/2}$}
\label{fig:2}
\end{subfigure}
    \\
\begin{subfigure}[b]{0.45\textwidth}
\includegraphics[width=\textwidth]{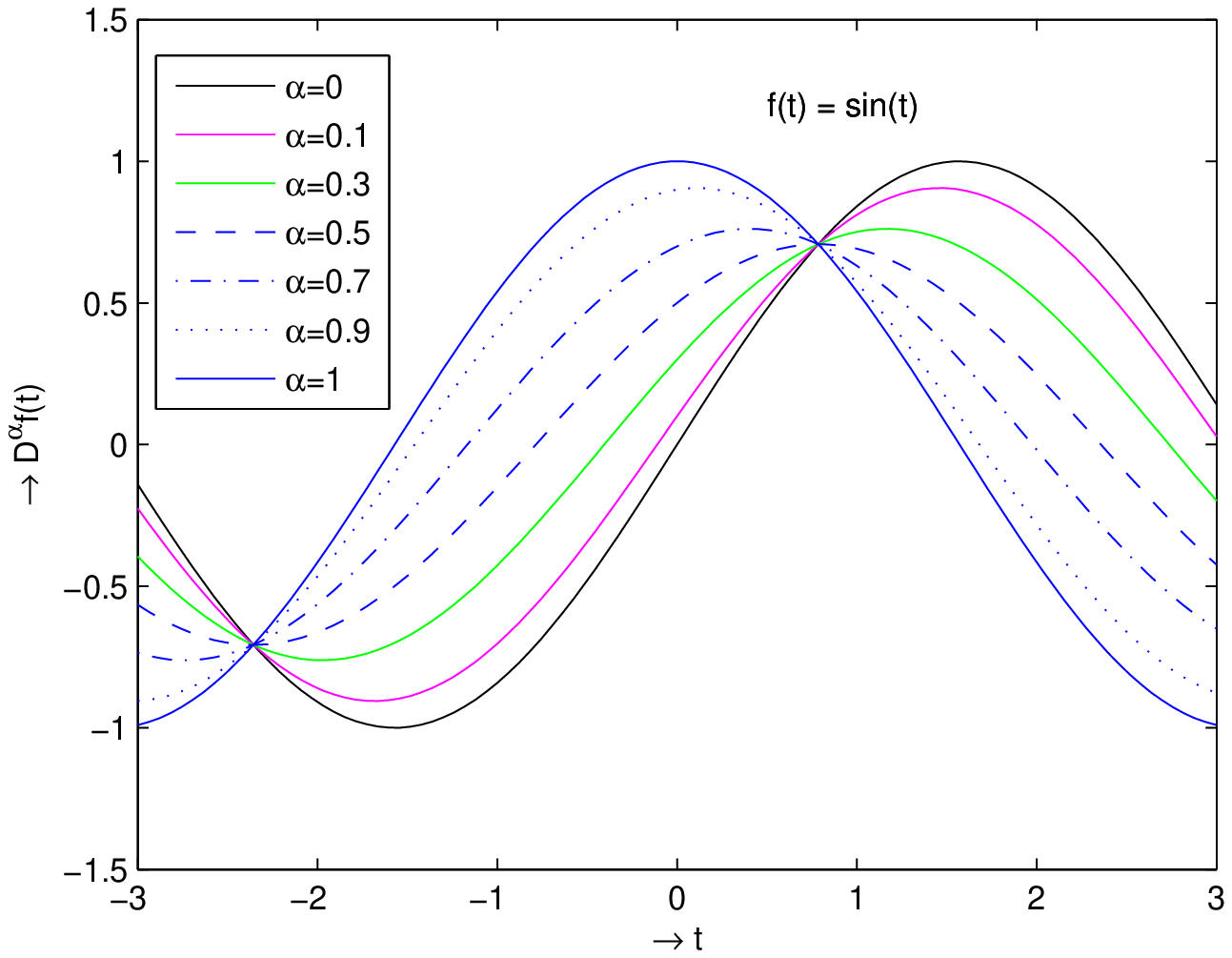}
\caption{$D^\alpha$ operating on  $\sin t$}
\label{fig:3}
\end{subfigure}
~ 
\begin{subfigure}[b]{0.45\textwidth}
\includegraphics[width=\textwidth]{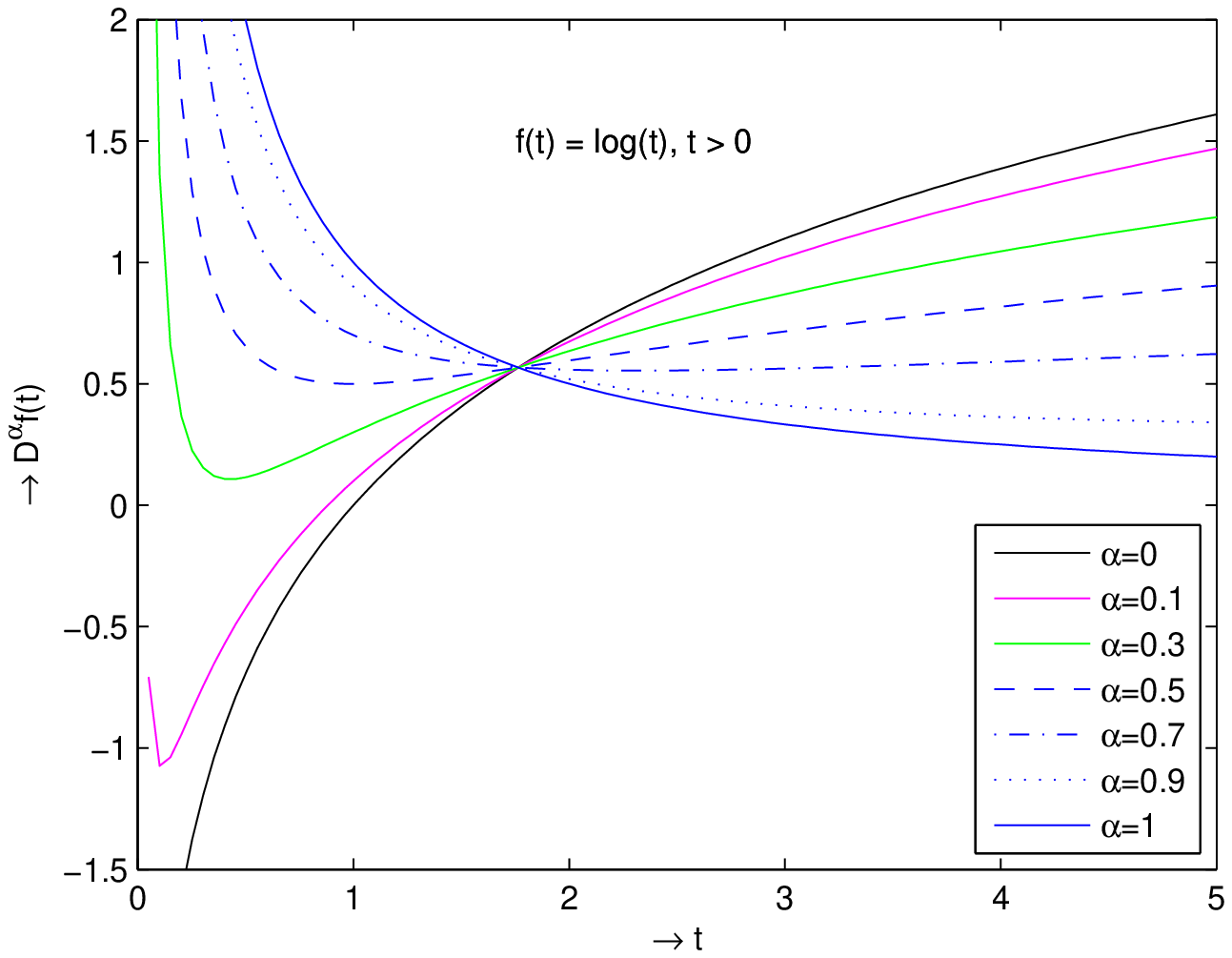}
\caption{$D^\alpha$ operating on  $\log t, t > 0$}
\label{fig:4}
\end{subfigure}
\caption{$D^\alpha$-derivative of different functions, $\alpha \in (0,1)$}
\end{figure}
    
\section{\textbf {Some Useful Theorems on deformable Derivative}}
In this section we extend  $ Rolle's $, Mean Value and Taylor's theorems to deformable  derivative with respect to $\alpha$.

\begin{theorem}\label{Rolle}{\bf\text{(Rolle's theorem on deformable derivative)}}
Let $f:[a,b]\longrightarrow \R$ be a function satisfying:
\begin{itemize}
\item [$($\rm i$)$]	$f$ is continuous on $[a,b]$
\item [$($\rm ii$)$] $f$ is $\alpha$-differentiable in $(a,b)$ 
\item [$($\rm iii$)$] $f(a)=f(b).$
\end{itemize}
Then, there exists a point $c\in (a,b)$ such that $D^\alpha f(c)=\beta f(c)$.
\end{theorem}

\begin{proof}
By applying corollary \ref{diff.4}, $f$ is differentiable in $(a,b)$. Thus $f$ satisfies all conditions of classical Rolle's theorem. Then there is a point $c\in(a,b)$ such that $D f(c)=0$. Hence using equation (**) of theorem \ref{diff.1}, we have $D^\alpha f(c)=\beta f(c)$.
\end{proof}	

Mean Value theorem is a consequence of Rolle's theorem so is case with deformable derivative. 	

\begin{theorem}\label{meanvalue}
{\bf\text{(Mean Value theorem on deformable derivative)}}
Let $f:[a,b]\longrightarrow R$ be a function satisfying:
\begin{itemize}
\item [$($\rm i$)$] $f$ is continuous on $[a,b]$
\item [$($\rm ii$)$] $f$ is $\alpha$- differentiable in $(a,b)$.  
\end{itemize}
Then, there exists a point $c\in (a,b)$ such that 
\[
D^\alpha f(c)=\beta f(c)+\alpha\frac{f(b)-f(a)}{b-a}.
\]
\end{theorem}
	
\begin{proof}
We consider a function $g$ defined by:
\[
g(x)=f(t)-f(a)-\frac{f(b)-f(a)}{b-a}t.
\]
Notice that $g(t)$ satisfies all the conditions of Rolle's Theorem \ref{Rolle}, there exists $c\in(a,b)$ such that $D^\alpha g(c)=\beta g(c).$ This yields the desired expression in the theorem.
\end{proof} 
	
\begin{theorem}\label{Taylor}
{\bf\text{(Taylor's theorem)}}	
Suppose $f$ is $n$-times $\alpha$-differentiable such that all $\alpha$-derivatives are continuous on $[a,a+h].$  Then
\[
f(a+h)=\sum\limits_{k=0}^{n-1}\frac{h^{k}}{k!\alpha^{k}}\left(D^{\alpha}_{k}f(a)-\beta\frac{(1-\theta)^{k-n+1}h}{\alpha n}D^{\alpha}_{k}f(a+\theta h)\right)+\frac{h^{n}}{n!\alpha }D^{\alpha}_{n}f(a+\theta h)
\]
where $D_k^\alpha= D^\alpha D^\alpha\cdots D^\alpha$, $(k$-times$)$,\; $0<\theta<1$.
\end{theorem}

\begin{proof}
Consider a function $\phi$ defined by:
\begin{equation}
\phi(t)=\sum\limits_{k=0}^{n-1}\frac{(a+h-t)^k}{k!\alpha^k}D^{\alpha}_k f(t) + \frac{A}{n!\alpha^{n}}(a+h-t)^{n},
\end{equation}
where $A$ is a constant to be chosen $A$ such that $\phi(a+h)=\phi(a)$. This yields 
\begin{equation}
\frac{A}{n!\alpha^{n}}h^{n}=f(a+h)-\sum\limits_{k=0}^{n-1} \frac{h^{k}}{k!\alpha^{k}}D^{\alpha}_{k}f(a)
\end{equation}
Now by hypothesis, $\phi$ is  $\alpha$-differentiable in $(a,a+h)$. Using part (d) of theorem \ref{properties.1}, the $\alpha$-derivative $D^{\alpha}\phi$ is given by
\begin{equation}
D^{\alpha}\phi(t)=\frac{(a+h-t)^{n-1}}{\alpha^{n-1}(n-1)!}D^{\alpha}_{n}f(t)+\frac{A}{\alpha^{n}n!}\left(\beta(a+h-t)^{n}-\alpha n(a+h-t)^{n-1}\right)
\end{equation}
Hence $\phi$ satisfies all the conditions of Rolle's Theorem \ref{Rolle}. So there is some $\theta\in(0,1)$ such that 
\[
D^{\alpha}\phi(a+\theta h)=\beta\phi(a+\theta h).
\]
Using equations ($1$), ($2$) and ($3$), we have 
\[
f(a+h)=\sum\limits_{k=0}^{n-1}\frac{h^{k}}{k!\alpha^{k}}\left(D^{\alpha}_{k}f(a)-\beta\frac{(1-\theta)^{k-n+1}h}{\alpha n}D^{\alpha}_{k}f(a+\theta h)\right)+\frac{h^{n}}{n!\alpha }D^{\alpha}_{n}f(a+\theta h)
\]
This completes the theorem.
\end{proof}

\section{\textbf{Fractional Integral}}
\noindent Fractional integral, being an inverse operator to fractional derivative, plays equally important role as that of fractional derivative in the field of fractional calculus. The section defines a  fractional integral as an inverse operator for deformable derivative. Some basic properties of this fractional integral are also discussed. All functions considered in this section are assumed to be continuous.\\

\begin{definition} 
Let $f$ be a continuous function defined on $[a,b]$. We define {\it $\alpha$-fractional Integral} of $f,~$ denoted by $I^{\alpha}_{a}f,$ by the integral
\begin{equation*}
I^{\alpha}_{a}f(t)=\frac{1}{\alpha}e^{\frac{-\beta }{\alpha}t}\int_{a}^{t}e^{\frac{\beta }{\alpha}x}f(x)dx,\quad\mbox{where}\;\alpha+\beta=1,\;\alpha\in(0,1].\tag{\#}    
\end{equation*}
\end{definition}

Some basic properties of this fractional integral are contained in next theorem.

\begin{theorem}\label{properties.2}
The operator $I^{\alpha}_a$ possesses the following properties:
\begin{itemize}
\item [$($\rm a$)$] Linearity: $I^{\alpha}_{a}\left(b f+c g\right)=b I^{\alpha}_{a}f+c I^{\alpha}_{a}g.$
\item [$($\rm b$)$]  Commutativity: 
$I_{a}^{\alpha_1}I_{a}^{\alpha_2}=I_{a}^{\alpha_2}I_{a}^{\alpha_1}$,
 where   $\alpha_i+\beta_i=1$,\;$i=1,2.$
\end{itemize}
\end{theorem}
\begin{proof}
Linearity readily follows from definition (\#). For commutativity, we consider
\begin{eqnarray*}
I_{a}^{\alpha_{1}}I_{a}^{\alpha_{2}}f(t) &=& I_{a}^{\alpha_{1}}\left(\frac{1}{\alpha_{2}}e^{-\frac{\beta_{2}}{\alpha_{2}}t}\int\limits_{a}^{t}e^{\frac{\beta_{2}}{\alpha_{2}}\theta}f(\theta)d\theta\right)\\ 
&=& \frac{1}{\alpha_{1}}e^{-\frac{\beta_{1}}{\alpha_{1}}t}\int\limits_{a}^{t}e^{\frac{\beta_{1}}{\alpha_{1}}x}\left(\frac{1}{\alpha_{2}}e^{-\frac{\beta_{2}}{\alpha_{2}}x}\int\limits_{a}^{x}e^{\frac{\beta_{2}}{\alpha_{2}}\theta}f(\theta)d\theta\right)dx\\
&=& \frac{1}{\alpha_{1}\alpha_{2}}e^{-\frac{\beta_{1}}{\alpha_{1}}t}\int\limits_{a}^{t}\int\limits_{a}^{x}e^{(\frac{\beta_{1}}{\alpha_{1}}-\frac{\beta_{2}}{\alpha_{2}})x}e^{\frac{\beta_{2}}{\alpha_{2}}\theta}f(\theta)\; d\theta dx\\
&=& \frac{1}{\alpha_{1}\alpha_{2}}e^{-\frac{\beta_{1}}{\alpha_{1}}t}\int\limits_{a}^{t}\int\limits_{\theta}^{t}e^{(\frac{\beta_{1}}{\alpha_{1}}-\frac{\beta_{2}}{\alpha_{2}})x}e^{\frac{\beta_{2}}{\alpha_{2}}\theta}f(\theta)\;dx d\theta\\
&=& \frac{1}{\alpha_{1}\alpha_{2}}e^{-\frac{\beta_{1}}{\alpha_{1}}t}\int\limits_{a}^{t}e^{\frac{\beta_{2}}{\alpha_{2}}\theta}f(\theta)\left(\int\limits_{\theta}^{t}e^{(\frac{\beta_{1}}{\alpha_{1}}-\frac{\beta_{2}}{\alpha_{2}})x}dx\right)d\theta\\
&=&\frac{1}{\beta_{1}\alpha_{2}-\beta_{2}\alpha_{1}}\left(e^{-\frac{\beta_{2}}{\alpha_{2}}t}\int\limits_{a}^{t}e^{\frac{\beta_{2}}{\alpha_{2}}\theta}f(\theta)d\theta-e^{-\frac{\beta_{1}}{\alpha_{1}}t}\int\limits_{a}^{t}e^{\frac{\beta_{1}}{\alpha_{1}}\theta}f(\theta)d\theta\right)\\
&=& \frac{1}{\beta_{1}\alpha_{2}-\beta_{2}\alpha_{1}}\left(\alpha_{2}I^{\alpha_{2}}_{a}-\alpha_{1}I^{\alpha_{1}}_{a}\right)f(t)
\end{eqnarray*}
Interchanging the role of $\alpha_1$ and $\alpha_2$,  we have
\[
I_{a}^{\alpha_{2}}I_{a}^{\alpha_{1}}f(t)=			\frac{1}{\beta_{2}\alpha_{1}-\beta_{1}\alpha_{2}}\left(\alpha_{1}I^{\alpha_{1}}_{a}-\alpha_{2}I^{\alpha_{2}}_{a}\right)f(t)
=I_{a}^{\alpha_{1}}I_{a}^{\alpha_{
2}}f(t)
\]
This completes the proof
\end{proof}

The next theorem is a version of fundamental theorem of calculus that roughly says that fraction integral $I_a^\alpha$ is an inverse operation of $\alpha$-differentiation $D^\alpha$.

\begin{theorem}\label{FTC}
{\bf\text{(Inverse Property)}} Let $f$ be a continuous function defined on $[a,b].$ Then  $I_{a}^{\alpha}f$ is $\alpha$-differentiable in $(a,b)$. In fact, we have $D^{\alpha}\left(I_{a}^{\alpha}f(x)\right)=f(x).$ Conversely Suppose $g$ is a continuous anti-$\alpha$-derivative of $f$ over $(a,b)$, that is $g=D^\alpha f$. Then we have
\[
I^{\alpha}_{a}\left(D^\alpha f(t)\right)=I^{\alpha}_{a}\left(g(t)\right)=f(t)-e^{\frac{\beta}{\alpha}(a-t)}f(a).
\]
\end{theorem}	

\begin{proof} 
Since $f$ is given to be continuous so in view of theorem \ref{diff.4}, $I^\alpha_a f$ is $\alpha$-differentiable. If we set $g=I^\alpha_a f$ then we have
\[
D^{\alpha}\left(I_{a}^{\alpha}f(t)\right)=D^\alpha g(t)=\alpha Dg(t)+\beta g(t)
\]
We know that a particular solution of the differential equation: $\alpha Dg+\beta g=f$ is given as 
\[
g(t)=\frac{1}{\alpha}e^{\frac{-\beta }{\alpha}t}\int_{a}^{t}e^{\frac{\beta }{\alpha}x}f(x)dx.
\]
Thus first part of theorem is complete. For the second part, we have 
\begin{eqnarray*}
g(t)&=& D^\alpha f(t)=\alpha Df(t)+\beta f(t)\\
\Rightarrow\quad I^\alpha_a g(t) &=& \alpha I^\alpha_a\left( Df(t) \right)+\beta I^\alpha_a f(t)\\
&=& e^{\frac{-\beta }{\alpha}t} \int_{a}^{t}e^{\frac{\beta }{\alpha}x}f^\prime(x)dx +\beta I^\alpha_a f(t)\\
&=& e^{\frac{-\beta }{\alpha}t}\left( \left[e^{\frac{\beta }{\alpha}x} f(x) \right]_a^t-\frac{\beta}{\alpha}\int_{a}^{t}e^{\frac{\beta }{\alpha}x}f(x)dx\right)
+\beta I^\alpha_a f(t)\\
\Rightarrow\quad I^\alpha_a g(t) &=& f(t)- e^{\frac{\beta }{\alpha}(a-t)}f(a).
\end{eqnarray*}
This completes the second part.
\end{proof}

We end up the section with a list of fractional integrals of some elementary functions in the following proposition and leave their verification for reader.

\begin{proposition}$~$
\begin{enumerate}
\item $I^{\alpha}_a\sin  t=\frac{1}{\alpha^2+\beta^2}\left(\beta \sin t-\alpha\cos t + e^{\frac{\beta}{\alpha}(a-t)} \left(\alpha\cos a -\beta\sin a\right)\right)$
\item $I^{\alpha}_a e^{t}= \left(e^{t} - e^{\frac{(a-\beta t)}{\alpha}}\right)$
\item $I^{\alpha}_a\lambda =\frac{\lambda}{\beta} \left(1-e^{\frac{\beta}{\alpha}(a-t)}\right)$where $\lambda$ is a constant.
\item $I^{\alpha}_0 t^n=\frac{1}{\beta}\left(\sum_{k=0}^{n}\frac{(-1)^k\; n!}{(n-k)!}\left(\frac{\alpha}{\beta}\right)^k t^{n-k} + (-1)^{n+1}\; n!\left(\frac{\alpha}{\beta}\right)^n e^{-\frac{\beta}{\alpha}t}\right)$
\end{enumerate}
\end{proposition}

\section{\textbf{Applications to Fractional Differential Equations}}
We solve some simple linear fractional differential equations using  deformable derivative as $D^\alpha$ operator. In first example we discuss method of solving  homogeneous linear, while in second non-homogeneous linear fractional differential equations.\\
 
\begin{example} Consider the fractional differential equation:  
\[
D^{\alpha} y(t)+P(t) y(t)=0,
\] 
where $P(t)$ is continuous. Using expression (**), the equation get transformed to 
\begin{eqnarray*}
\alpha Dy+ \beta y+P(t) y &=& 0\\
\Rightarrow\quad Dy+\frac{\left(\beta+P(t)\right)}{\alpha}y &=& 0.
\end{eqnarray*}
Which is simple first order linear ordinary differential equation  whose general solution is given by 
\[
y = C e^{\displaystyle{\frac{-\left(\beta t+\int P(t)dt\right)}{\alpha}}}
\]
where $C$ is arbitrary constant.
\end{example}

\begin{example} We now consider a non-homogeneous linear fractional equation:
\[
D^{1/2}y+y=te^{-t}.
\] 
This can be written as 
\[
\frac{1}{2}y+\frac{1}{2}Dy + y = te^{-t}\quad
\Rightarrow \quad Dy + 3y = 2te^{-t}
\]
Whose general solution is given by 
\[
y(t)=Ce^{-3t}+\left (t-\frac{1}{2}\right)e^{-t}
\]
where $C$ is a constant.
\end{example}

\begin{example} The fractional differential equation :
\[
D^{\alpha_{2}}[D^{\alpha_{1}}y(t)]=0
\]
is equivalent to the following second order homogeneous differential equation:
\[
\alpha_{1}\alpha_{2}D^2y+(\alpha_{1}\beta_{2}+\alpha_{2}\beta_{1})Dy+\beta_{1}\beta_{2}y=0.
\]
The roots its auxiliary equation are:
\[
\frac{-\beta_{1}}{\alpha_{1}}\; \mbox{and}\;\frac{-\beta_{2}}{\alpha_{2}}.
\]
Hence in case of distinct roots the general solution of the fractional differential equation is
\[
y=C_1e^{\displaystyle{-\frac{\beta_1}{\alpha_1}t}}+C_2e^{\displaystyle{-\frac{\beta_2}{\alpha_2}t}},
\]
and in case of repeated roots, we have
\[
y=(C_1+C_2t)e^{\displaystyle{-\frac{\beta}{\alpha}t}}
\]
\end{example}

We end up the paper with some important questions that are yet to be answered.
\begin{itemize}
\item [$($\rm i$)$]	What is geometric interpretation and physical significance of the deformable derivative ?
\item [$($\rm ii$)$]  Is there any similarity between the classical  fractional derivative and deformable derivative ?
\item [$($\rm iii$)$] Deformable derivative is equivalent to ordinary one but not same so it could be used to analyze function.    
\end{itemize}

\bibliographystyle{amsplain}
\bibliography{document2}
\end{document}